\documentclass[11pt]{amsart}
\usepackage{amsmath,amssymb,,latexsym,esint,cite}
\usepackage{verbatim}

\usepackage[left=3.1cm,right=3.1cm,top=3.2cm,bottom=3.2cm]{geometry}

\usepackage{color,enumitem,graphicx}
\usepackage[colorlinks=true,urlcolor=blue, citecolor=red,linkcolor=blue,linktocpage,pdfpagelabels, bookmarksnumbered,bookmarksopen]{hyperref}

\usepackage[hyperpageref]{backref}
\usepackage[english]{babel}

\usepackage[colorinlistoftodos]{todonotes}

\makeatletter
\providecommand\@dotsep{5}
\def\listtodoname{List of Todos}
\def\listoftodos{\@starttoc{tdo}\listtodoname}
\makeatother

\newcommand{\R}{{\mathbb R}}
\newcommand{\N}{{\mathbb N}}

\renewcommand{\O}{{\mathcal O}}
\renewcommand{\S}{{\mathcal S}}
\newcommand{\M}{{\mathcal M}}
\renewcommand{\phi}{{\varphi}}

\newcommand{\eps}{{\varepsilon}}
\renewcommand{\epsilon}{{\varepsilon}}
\renewcommand{\theta}{{\vartheta}}

\numberwithin{equation}{section}

\newtheorem{theorem}{Theorem}[section]
\newtheorem{lemma}[theorem]{Lemma}
\newtheorem{corollary}[theorem]{Corollary}
\newtheorem{proposition}[theorem]{Proposition}

\newtheorem{definition}[theorem]{Definition}

\begin{document}

\title{The Brezis-Nirenberg problem for nonlocal systems}

\author[L.F.O.\ Faria]{L.F.O.\ Faria}
\author[O.H.\ Miyagaki]{O.H.\ Miyagaki}
\author[F.R.\ Pereira]{F.R.\ Pereira}
\author[M.\ Squassina]{M.\ Squassina}
\author[C.\ Zhang]{C.\ Zhang}

\address[L.F.O. Faria, O.H.\ Miyagaki, F.R. Pereira]{Departamento de Matem\'atica
\newline \indent
Universidade Federal de Juiz de Fora
\newline\indent
Juiz de Fora, CEP 36036-330, Minas Gerais, Brazil}
\email{luiz.faria@ufjf.edu.br,ohmiyagaki@gmail.com,fabio.pereira@ufjf.edu.br}

\address[M. Squassina]{Dipartimento di Informatica
\newline \indent
Universit\`a degli Studi di Verona
\newline\indent
Strada Le Grazie 15, I-37134 Verona, Italy}
\email{marco.squassina@univr.it}

\address[C.\ Zhang]{Chern Institute of Mathematics 
	\newline \indent
	Nankai University
	\newline\indent
	 Tianjin 300071, China}
\email{1120140009@mail.nankai.edu.cn}


\thanks{L.F.O.\ Faria was supported by Fapemig CEX APQ-02406-12, 
	O.H.\ Miyagaki by CNPq/Brazil and
	CAPES/Brazil (Proc 2531/14-3) and 
	F.R.\ Pereira by Fapemig/Brazil (CEX APQ 00972/13) and 
	by the program CAPES/Brazil (Proc 99999.007090/2014-05).
	Moreover, M.\ Squassina was partially supported by PROPG/UFJF-Edital 02/2014.\ 
	This paper was partly developed 
	while the second author was visiting the Department of Mathematics of Rutgers University and the fourth author was visiting the University of Juiz de Fora. The hosting  institutions  are  gratefully  acknowledged.}


\begin{abstract}
By means of variational methods
we investigate existence, non-existence as well as regularity of weak solutions
for a system of nonlocal equations involving the fractional laplacian
operator and with nonlinearity reaching the critical growth and interacting, in
a suitable sense, with the spectrum of the operator.
\end{abstract}

\maketitle

\begin{center}
	\begin{minipage}{9.5cm}
		\small
		\tableofcontents
	\end{minipage}
\end{center}

%

\section{Introduction and results}
\noindent
Let $\Omega\subset\R^{N}$ be a
smooth bounded domain.\ In 1983, Br\'ezis and Nirenberg, in the seminal paper \cite{bn},
showed that the critical growth semi-linear problem
\begin{equation}\label{eq0.0}
\left\{
\begin{array}{ll}
-\Delta u= \lambda u + u^{2^*-1} & \; \mbox{in $\Omega$}, \\
u>0 &\;\mbox{in $\Omega$}, \\
u=0 &\; \mbox{on $\partial\Omega$}, \\
\end{array}
\right.
\end{equation}
admits a solution provided that $\lambda\in (0,\lambda_1)$ and $N\geq 4$,
$\lambda_1$ being the first eigenvalue of $-\Delta$ with homogeneous 
Dirichlet boundary conditions
and $2^* = 2N/(N-2)$ the critical Sobolev exponent. Furthermore, in dimension $N=3$, the same existence result holds provided that $\mu< \lambda<\lambda_1$, for a suitable $\mu>0$ (if $\Omega$ is a ball, then $\mu=\lambda_1/4$ is sharp). By Poho\v zaev identity, if $\lambda \not\in (0,\lambda_1)$ and $\Omega$ is a star-shaped domain, then problem \eqref{eq0.0}
admits no solution. Later on, in 1984, Cerami, Fortunato and Struwe
obtained in \cite{cfs} multiplicity results for the nontrivial solutions of 
\begin{equation}\label{eq0.0-sign}
\left\{
\begin{array}{ll}
-\Delta u= \lambda u + u^{2^*-1} & \; \mbox{in $\Omega$}, \\
u=0 &\; \mbox{on $\partial\Omega$}, \\
\end{array}
\right.
\end{equation}
 when $\lambda$ belongs to a left neighborhood of an
eigenvalue of $-\Delta$.\ In 1985, Capozzi, Fortunato and Palmieri
proved in \cite{cfp} the existence of a nontrivial solution of \eqref{eq0.0-sign} for all
$\lambda >0$ and $N\geq 5$ or for $N\geq 4$ and $\lambda$ different 
from an eigenvalue of $-\Delta$. 
Let $s\in (0,1)$ and  $N>2s$. The aim of 
this paper is to obtain a Br\'ezis-Nirenberg type result  
for the following fractional system
\begin{equation}\label{eq:1.1}
\left\{
\begin{aligned}
(-\Delta)^s u &=au+bv+\frac{2p}{p+q}|u|^{p-2}u |v|^q   && \text{in $\Omega$,}\\
(-\Delta)^s v &=bu+cv+\frac{2q}{p+q}|u|^p |v|^{q-2}v   && \text{in $\Omega$,}\\
 u=&\, v=0 && \text{in $\R^N\setminus\Omega$,}
\end{aligned}
\right.
\end{equation}
where $(-\Delta)^{s}$ is defined, on smooth functions, by
\[
(-\Delta)^su(x):=C(N,s)\lim_{\eps\searrow 0}\int_{\R^N\setminus B_\eps(x)}\frac{u(x)-u(y)}{|x-y|^{N+2s}}\,dy, \quad\,\, x\in\R^N,
\]
$C(N,s)$ being a suitable positive constant and $p,q>1$ are
such that $p+q$ is compared to $ 2^*_s:=2N/(N-2s)$,
the fractional critical Sobolev  exponent \cite{nezza}.
The corresponding system in the local case was studied in \cite{alves}.
For positive solutions, system \eqref{eq:1.1} turns into
\begin{equation}\label{eq:1.1+}
\left\{
\begin{aligned}
(-\Delta)^s u &=au+bv+\frac{2p}{p+q}u^{p-1} v^q   && \text{in $\Omega$,}\\
(-\Delta)^s v &=bu+cv+\frac{2q}{p+q}u^p v^{q-1}   && \text{in $\Omega$,}\\
 u>0&,\,\, v>0   && \text{in $\Omega$,} \\
 u=&\, v=0 && \text{in $\R^N\setminus\Omega$.}
\end{aligned}
\right.
\end{equation}
In the following we shall assume that
$\Omega$ is a smooth bounded domain of $\R^N$ with $N>2s$
and we shall denote by $(\lambda_{i,s})$ the 
sequence of eigenvalues of $(-\Delta)^s$ with
homogeneous Dirichlet type boundary condition and by $\mu_1$ and $\mu_2$ the real eigenvalues of the matrix
$$
A:=\left(
                               \begin{matrix}
                                 a & b \\
                                 b & c \\
                               \end{matrix}
                             \right), \,\,\quad a,b,c\in\R.
 $$
 Without loss of generality, we will assume $\mu_1\leq \mu_2$.\ 
 By solution we shall always mean {\em weak} 
 solution in the sense specified in Section~\ref{preliminar}, where the functional
 space $X(\Omega)$ is fully described.
 It is known that the first eigenvalue
 $\lambda_{1,s}$ is positive, simple and characterized by
 \begin{equation}
 \label{lambda1}
 \lambda_{1,s}=\inf_{u\in X(\Omega)\setminus\{0\}}
 \frac{\displaystyle\int_{\R^{N}}|(-\Delta)^{\frac{s}{2}} u|^2dx}{\displaystyle\int_{\R^{N}}|u|^{2}dx}.
 \end{equation}
\vskip6pt
\noindent
The following are the main results of the paper.

\begin{theorem}[Existence I]\label{teo1}
Assume that $b\geq 0$, $\mu_2<\lambda_{1,s}$ and 
$$
p+q<2^{*}_{s}. 
$$
Then \eqref{eq:1.1+} admits a solution.
\end{theorem}


\begin{theorem}[Existence II]
\label{teo2}
Assume that $b\geq 0$ and 
$$
p+q=2^{*}_{s}.
$$ 
Then the following facts hold.
\begin{enumerate}
\item If $N\geq 4s$ and $0<\mu_1\leq \mu_2<\lambda_{1,s}$ then \eqref{eq:1.1+} admits a solution; 
\item If $2s<N<4s$, there is $\bar{\mu}>0$ such that \eqref{eq:1.1+} admits a solution if $\bar{\mu}<\mu_1\leq \mu_2<\lambda_{1,s}$.
\end{enumerate}
\end{theorem}


\begin{theorem}[Nonexistence]
\label{teo3}
Assume that $p+q=  2^{*}_{s}$ and one of the following facts hold.
\begin{enumerate}
	\item $\Omega$ is star-shaped with respect to the origin and $\mu_2 < 0$;
	\item $\Omega$ is star-shaped with respect to the origin and $A$ is the zero matrix;
	\item $b\geq 0$ and $\mu_1\geq \lambda_{1,s}-|a-c|$.
\end{enumerate}
Then \eqref{eq:1.1+}  does not admit any solution. Furthermore, if $\mu_2\leq 0$ and 
$$
p+q>  2^{*}_{s},
$$
then \eqref{eq:1.1+}  does not admit any bounded solution if $\Omega$ is star-shaped with respect to the origin.
\end{theorem}

\begin{theorem}[Regularity]
	\label{regul}
	Assume that $p+q\leq 2^*_s.$
	 If $(u,v)$ is a solution to \eqref{eq:1.1}, then $u,v\in C^{1,\alpha}_{{\rm loc}}(\Omega)$ for $s\in (0,1/2)$ and
	 $u,v\in C^{2,\alpha}_{{\rm loc}}(\Omega)$ for $s\in (1/2,1)$.
	 In particular $(u,v)$ solves \eqref{eq:1.1} in classical sense.
\end{theorem}

\noindent
The nonexistence result stated in (3) of Theorem~\ref{teo3} 
holds in any bounded domain.\ For $b=0$ it reads as 
$\mu_2\geq \lambda_{1,s},$ properly complementing the 
assertions of Theorem~\ref{teo2}.\
The above results provide a full extension of the classical results 
of Br\'ezis and Nirenberg \cite{bn} for the local case $s=1$.\
We point out that we adopt in the paper the integral definition of the fractional
laplacian in a bounded domain and we do not exploit any {\em localization
procedure} based upon the Caffarelli-Silvestre extension \cite{cafsilv}, as done e.g.\
in \cite{colorado}.\
See \cite{confronto} for a nice comparison between these two different notions
of fractional laplacian in bounded domains.
By choosing $p=q=2^*_s/2$, system \eqref{eq:1.1+} reduces to
 \begin{equation}\label{eq:1.1++}
 \left\{
 \begin{aligned}
 (-\Delta)^s u &=au+bv+u^{2s/(N-2s)} v^{N/(N-2s)}   && \text{in $\Omega$,}\\
 (-\Delta)^s v &=bu+cv+u^{N/(N-2s)} v^{2s/(N-2s)}   && \text{in $\Omega$,}\\
 u>0&,\,\, v>0   && \text{in $\Omega$,} \\
 u=&\, v=0 && \text{in $\R^N\setminus\Omega$,}
 \end{aligned}
 \right.
 \end{equation}
which, in the particular case of $a=c$, setting $u=v$, 
boils down to the scalar equation
 \begin{equation}
 \left\{
 \begin{aligned}
 (-\Delta)^s u &=\lambda u+u^{2^*_s-1}   && \text{in $\Omega$,}\\
  u>0&   && \text{in $\Omega$,} \\
  u=0 &&& \text{in $\R^N\setminus\Omega$,}
 \end{aligned}
 \right.
 \end{equation}
 which is the natural fractional counterpart
 for the classical Br\'ezis-Nirenberg problem \cite{bn}.
For existence results for this problem, we refer 
to \cite{servadeiTAMS,low} and to the references therein.



\section{Preliminary stuff}
\label{preliminar}

\subsection{Notations and setting}
We refer the reader to \cite{nezza} 
for further details about the functional framework that follows.\
For any measurable function $u:\R^N\to\R$ we define the Gagliardo seminorm by setting
\[
[u]_{s}:=\Big(\frac{C(N,s)}{2}\int_{\R^{2N}}\frac{|u(x)-u(y)|^2}{|x-y|^{N+2s}} dxdy\Big)^{1/2}=
\Big(\int_{\R^{N}}|(-\Delta)^{\frac{s}{2}} u|^2dx\Big)^{1/2}.
\]
The second equality follows by \cite[Proposition 3.6]{nezza} when 
the above integrals are finite. Then, we introduce the fractional Sobolev space
\[
H^s(\R^N)=\{u\in L^2(\R^N):\,\, [u]_{s}<\infty\},\quad
\|u\|_{H^s}=(\|u\|^2_{L^2}+[u]_s^2)^{1/2},
\]
which is a Hilbert space and we consider the closed subspace
\begin{equation}
\label{defX}
X(\Omega):=\{u\in H^s(\R^N):\,u=0\,\,\,\mbox{a.e. in $\R^N\setminus\Omega$}\}.
\end{equation}
Due to the fractional Sobolev inequality, $X(\Omega)$ is a Hilbert space with inner product
\begin{equation}
\label{innprod}
\langle u,v\rangle_X:=\frac{C(N,s)}{2}\int_{\R^{2N}}\frac{(u(x)-u(y))(v(x)-v(y))}{|x-y|^{N+2s}}\, dx\, dy,
\end{equation}
which induces the norm $\|\cdot\|_X=[\,\cdot\,]_{s}$.
Now, we consider the Hilbert space given by the product
\begin{equation}
\label{defY}
Y(\Omega):=X(\Omega)\times X(\Omega),
\end{equation}
equipped with the inner product
\begin{equation}
\label{innprodY}
\langle (u,v),(\varphi,\psi)\rangle_Y:=\langle u,\varphi\rangle_X+\langle v,\psi\rangle_X
\end{equation}
and the norm
\begin{equation}
\label{normY}
\|(u,v)\|_{Y}:=(\|u\|_{X}^2+\|v\|_{X}^2)^{1/2}.
\end{equation}
We shall consider $L^m(\Omega)\times L^m(\Omega)$ ($m>1$) 
equipped with the standard product norm
\begin{equation}
\label{normL}
\|(u,v)\|_{L^m\times L^m}:=(\|u\|_{L^m}^2+\|v\|_{L^m}^2)^{1/2}.
\end{equation}
We recall that we have
\begin{equation}
\label{controllo}
\mu_1 |U|^2\leq (AU,U)_{\R^2}\leq \mu_2 |U|^2,\quad\text{for all $U:=(u,v) \in\R^2$.}
\end{equation}
In this paper, we consider the following notation for product space $ \mathfrak{F} \times \mathfrak{F} := \mathfrak{F}^2$
and set
$$
w^+ (x):=\max \{w(x),0\},\quad\,\,  w^- (x):=\min\{w(x),0\},
$$ 
for positive and negative part of a function $w.$  Consequently we get $w=w^+ + w^- $.\
During chains of inequalities, universal constants will be denoted by the
same letter $C$ even if their numerical value may change from line to line.

\subsection{Weak solutions}
Consider the following system
\begin{equation}\label{eq:1.aux}
\left\{
\begin{aligned}
(-\Delta)^s u &= f(u,v)   && \text{in $\Omega$,}\\
(-\Delta)^s v &= g(u,v)   && \text{in $\Omega$,}\\
 u=&\, v=0 && \text{in $\R^N\setminus\Omega$.}
\end{aligned}
\right.
\end{equation}
where $f,g:\R\times\R\to\R$ are
Carath\'eodory mappings  which satisfies, respectively, the following growths conditions
\begin{align}\label{ff}
& |f(z,w)|\le C(1+|z|^{2^*_s-1}+|w|^{2^*_s-1}), \quad
\mbox{for all $(z,w)\in\mathbb{R}^2$}, \\
& \label{gg}
|g(z,w)|\le C(1+|z|^{2^*_s-1}+|w|^{2^*_s-1}), \quad
\mbox{for all $(z,w)\in\mathbb{R}^2$}.
\end{align}
\vskip2pt

\begin{definition}
	\label{defweak}
We say that $(u,v)\in Y(\Omega)$ is a weak solutions of (\ref{eq:1.aux}), if
\begin{equation}\label{weak}
\langle (u,v),(\varphi,\psi)\rangle_Y= \int_\Omega f(u,v)\varphi dx + \int_\Omega g(u,v)\psi dx,
\end{equation}
for all $(\varphi,\psi)\in Y(\Omega)$.
\end{definition}

\subsection{A priori bounds}\label{apb}
\noindent
We introduce some notation: for all $t\in\R$ and $k>0$, we set
\begin{equation}\label{trunk}
t_k:={\rm sgn}(t)\min\{|t|,k\}.
\end{equation}

\noindent
From \cite[Lemma 3.1]{IMS} we recall the following

\begin{lemma}\label{ineq}
For all $a,b\in\R$, $r\ge 2$, and $k>0$ we have
\[(a-b)(a|a|_k^{r-2}-b|b|_k^{r-2})\geq \frac{4(r-1)}{r^2}(a|a|_k^{\frac{r}{2}-1}-b|b|_k^{\frac{r}{2}-1})^2.\]
\end{lemma}

\noindent
In the following, we prove an $L^\infty$-bound on the weak solutions of \eqref{eq:1.aux}
which will be needed in order to get nonexistence and regularity results.

\begin{lemma}\label{linfty}
Assume that $f$ and $g$ satisfy \eqref{ff}-\eqref{gg} and let $(u,v)\in Y(\Omega)$ 
be a weak solution to \eqref{eq:1.aux}. Then we have $u,v\in L^\infty(\Omega)$.
\end{lemma}
\begin{proof}
For all $r\ge 2$ and $k>0$, the map $t\mapsto t|t|_k^{r-2}$ is Lipschitz in $\R$. Then 
 $$
 (u|u|_k^{r-2},0)\in Y(\Omega),\,\,\, (0,v|v|_k^{r-2})\in Y(\Omega). 
 $$ 
We  test equation \eqref{weak} with $(u|u|_k^{r-2},0)$,
we apply the fractional Sobolev inequality, Young's inequality, Lemma \ref{ineq}, 
and use \eqref{ff} to end up with
\begin{equation}\label{sc1}
\begin{split}
\|u|u|_k^{\frac{r}{2}-1}\|_{L^{2^*_s}}^2 &\le C\|u|u|_k^{\frac{r}{2}-1}\|_X^2\le \frac{Cr^2}{r-1}\langle u,u|u|_k^{r-2}\rangle_X \\
&\le Cr\int_\Omega |f(u,v)||u||u|_k^{r-2}\, dx \\
&\le Cr\int_\Omega\big(|u||u|_k^{r-2}+|u|^{2^*_s}|u|_k^{r-2}+|v|^{2^*_s-1}|u||u|_k^{r-2}\big)\, dx\\
&\le Cr\int_\Omega\big(|u|^{r-1}+|u|^{2^*_s+r-2}+|v|^{2^*_s+r-2}\big)\, dx,
\end{split}
\end{equation}
for some $C>0$ independent of $r\geq 2$ and $k>0$. Then, Fatou Lemma, as $k\to\infty$ yields
\begin{equation}\label{sc2}
\|u\|_{L^{\gamma^2 r}}^r\le Cr\Big(\int_\Omega\big(|u|^{r-1}+|u|^{2^*_s+r-2}+|v|^{2^*_s+r-2}\big)\, dx\Big)
\end{equation}
where  $\gamma=(2^*_s/2)^{1/2}$ (the right hand side may at this stage be $\infty$).
Now, in a similar way,  test \eqref{weak} with $(0,v|v|_k^{r-2})$ to obtain
for some $C>0$ independent of $r\geq 2$ 
\begin{equation}\label{sc2v}
\|v\|_{L^{\gamma^2 r}}^r\le Cr\Big(\int_\Omega\big(|v|^{r-1}+|u|^{2^*_s+r-2}+|v|^{2^*_s+r-2}\big)\, dx\Big),
\end{equation}
(the right hand side may be $\infty$).
By \eqref{sc2} and \eqref{sc2v} we get
\begin{equation}\label{sct}
\|u\|_{L^{\gamma^2 r}}^r+\|v\|_{L^{\gamma^2 r}}^r\le Cr\Big(\int_\Omega\big(|u|^{r-1}+|v|^{r-1}+|u|^{2^*_s+r-2}+|v|^{2^*_s+r-2}\big)\, dx\Big).
\end{equation}
Our aim is to develop a suitable bootstrap argument to prove that $u,v\in L^p(\Omega)$ for all $p\ge 1$.
We start from \eqref{sct}, with $r=2^*_s+1>2$, and fix $\sigma>0$ such that $Cr\sigma<\frac{1}{2}$.
Then there exists $K_0>0$ (depending on $u$ and $v$) such that
\begin{equation}
\label{piccolo}
\left(\int_{\{|u|> K_0\}}|u|^{2^*_s}\, dx\right)^{1-\frac{2}{2^*_s}}+\left(\int_{\{|v|> K_0\}}|v|^{2^*_s}\, dx\right)^{1-\frac{2}{2^*_s}}\leq\sigma.
\end{equation}
By H\"older inequality and \eqref{piccolo} we have
\begin{equation}\label{r1}\begin{array}{ll}
\displaystyle\int_\Omega |u|^{2^*_s+r-2} \, dx&\le \displaystyle K_0^{2^*_s+r-2}|\{|u|\le K_0\}|+\int_{\{|u|>K_0\}}|u|^{2^*_s+r-2}\, dx \\
&\displaystyle\le K_0^{2^*_s+r-2}|\Omega|+\Big(\int_\Omega (u^r)^{\frac{2^*_s}{2}}\, dx\Big)^{\frac {2}{2^*_s}}\Big(\int_{\{|u|>K_0\}}|u|^{2^*_s} \, dx\Big)^{1-\frac{2}{2^*_s}} \\
&\displaystyle\le K_0^{2^*_s+r-2}|\Omega|+\sigma\|u\|_{L^{\gamma^2 r}}^r \end{array}
\end{equation}
and
\begin{align}\label{r2}
\displaystyle\int_\Omega |v|^{2^*_s+r-2} \, dx\le\displaystyle  K_0^{2^*_s+r-2}|\Omega|+\sigma\|v\|_{L^{\gamma^2r}}^r.
\end{align}
\noindent
By \eqref{sct}, \eqref{r1} and \eqref{r2}, we have
\begin{equation}\label{rt}
\frac{1}{2}\left(\|u\|_{L^{\gamma^2 r}}^r+\|v\|_{L^{\gamma^2 r}}^r\right)\le Cr\Big(\int_\Omega\big(|u|^{r-1}+|v|^{r-1}\big)\, dx+ K_0^{2^*_s+r-2}\Big).
\end{equation}
Since $r=2^*_s+1$, we get $u,v\in L^{\frac{2^*_s(2^*_s+1)}{2}}(\Omega)$.
We define a sequence $\{r_n\}$ with
\begin{equation*}
r_0=2_s^*+1,\quad r_{n+1}=\gamma^2r_n-2_s^*+2.
\end{equation*}
Since 
$$
2_s^*+r_0-2<\displaystyle\frac{2^*_s(2^*_s+1)}{2},
$$ 
we get
$$
\|u\|_{L^{2_s^*+r_0-2}}+\|v\|_{L^{2_s^*+r_0-2}}<+\infty.
$$ 
Hence, we aim to
begin an iteration in order to get the $L^\infty$-bounds of $u$ and $v$.
Using formula \eqref{sct} and H\"{o}lder inequality, we obtain
\begin{align*}
& \|u\|_{L^{\gamma^2r}}+\|v\|_{L^{\gamma^2r}}  \\
&\leq (Cr)^{\frac{1}{r}}\Big(
|\Omega|^{\frac{2_s^*-1}{2_s^*+r-2}}\big(\|u\|_{L^{2_s^*+r-2}}^{r-1}+\|v\|_{L^{2_s^*+r-2}}^{r-1}\big)+
\big(\|u\|_{L^{2_s^*+r-2}}^{2_s^*+r-2}+\|v\|_{L^{2_s^*+r-2}}^{2_s^*+r-2}\big)\Big)^\frac{1}{r}  \\
&\leq (Cr)^{\frac{1}{r}}\Big(
\big(1+|\Omega|^{\frac{2_s^*-1}{2_s^*}}\big)\big(\|u\|_{L^{2_s^*+r-2}}+\|v\|_{L^{2_s^*+r-2}}\big)^{r-1}+
\big(\|u\|_{L^{2_s^*+r-2}}+\|v\|_{L^{2_s^*+r-2}}\big)^{2_s^*+r-2}\Big)^\frac{1}{r}.
\end{align*}
Substituting $r_{n+1}$ for $r$, since $\gamma^2r_n=2_s^*+r_{n+1}-2$, we get
\begin{align}\label{sct2}
&\|u\|_{L^{\gamma^2r_{n+1}}}+\|v\|_{L^{\gamma^2r_{n+1}}}\\
& \leq (Cr_{n+1})^{\frac{1}{r_{n+1}}}\Big( \notag
C\big(\|u\|_{L^{\gamma^2r_n}}+\|v\|_{L^{\gamma^2r_n}}\big)^{r_{n+1}-1}+
\big(\|u\|_{L^{\gamma^2r_n}}+\|v\|_{L^{\gamma^2r_n}}\big)^{\gamma^2r_n}\Big)^\frac{1}{r_{n+1}}
\end{align}
Denote
$$T_n:=\max\{1,\|u\|_{L^{\gamma^2r_n}}+\|v\|_{L^{\gamma^2r_n}}\}.$$
Then \eqref{sct2} can be written as
\begin{equation}
\label{indu}
T_{n+1}\leq (1+C)^{\frac{1}{r_{n+1}}}{r_{n+1}}^{\frac{1}{r_{n+1}}}{T_n}^{\frac{\gamma^2r_n}{r_{n+1}}}.
\end{equation}
Since $r_{n+1}=\gamma^2r_n-2_s^*+2$, by induction it is possible to prove that
\begin{equation*}
	\frac{r_{n+1}}{\gamma^{2n+2}}=2_s^*-1+2\gamma^{-2n-2},\quad n\in\N.
\end{equation*}
If $n=0$ the assertion follows by a direct calculation. Assume now that the assertion 
holds for a given $n\geq 1$ and let's prove it for $n+1$. We get 
\begin{align*}
\frac{r_{n+2}}{\gamma^{2n+4}}
&=\frac{r_{n+1}}{\gamma^{2n+2}}-\frac{2_s^*-2}{\gamma^{2n+4}} \\
&=2_s^*-1+2\gamma^{-2n-2}-\frac{2_s^*-2}{\gamma^{2n+4}} \\
&=2_s^*-1+2\gamma^{-2n-4},
\end{align*}
which proves the claim. In particular, $\frac{r_{n+1}}{\gamma^{2n+2}}\approx 2^*_s-1$.
From \eqref{indu}, we also have 
\begin{align*}
T_{n+1}&\leq (1+C)^{\frac{1}{r_{n+1}}}{r_{n+1}}^{\frac{1}{r_{n+1}}}{T_n}^{\frac{\gamma^2r_n}{r_{n+1}}}\\
&\leq (1+C)^{\frac{1}{r_{n+1}}}{r_{n+1}}^{\frac{1}{r_{n+1}}}
{\Big((1+C)^{\frac{1}{r_n}}{r_n}^{\frac{1}{r_n}}{T_{n-1}}^{\frac{\gamma^2r_{n-1}}{r_n}}\Big)}^{\frac{\gamma^2r_n}{r_{n+1}}}\\
&=(1+C)^\frac{1+\gamma^2}{r_{n+1}}{r_{n+1}}^{\frac{1}{r_{n+1}}}{r_n}^{\frac{\gamma^2}{r_{n+1}}}
{T_{n-1}}^{\frac{\gamma^4r_{n-1}}{r_{n+1}}}\leq \cdots\\ 
&\leq(1+C)^\frac{1+\gamma^2+\gamma^4+\cdots+\gamma^{2n}}{r_{n+1}}
\Big({r_{n+1}}^{\frac{1}{r_{n+1}}}{r_n}^{\frac{\gamma^2}{r_{n+1}}}{r_{n-1}}^{\frac{\gamma^4}{r_{n+1}}}\cdots
{r_1}^{\frac{\gamma^{2n}}{r_{n+1}}}\Big){T_0}^\frac{\gamma^{2n+2}r_0}{r_{n+1}}\\
&=(1+C)^\frac{\gamma^{2n+2}-1}{(\gamma^2-1)r_{n+1}}
\Big(\prod\limits_{i=0}^{n}r_{i+1}^{\gamma^{2(n-i)}}\Big)^{\frac{1}{r_{n+1}}}{T_0}^\frac{\gamma^{2n+2}r_0}{r_{n+1}}.
\end{align*}
We can easily compute that
\begin{equation*}
\frac{\gamma^{2n+2}-1}{(\gamma^2-1)r_{n+1}}\approx \frac{2}{(2_s^*-1)(2^*_s-2)},
\qquad\,\,
\frac{\gamma^{2n+2}r_0}{r_{n+1}}\approx \frac{2_s^*+1}{2^*_s-1}.
\end{equation*}
Moreover, $r_{i+1}<r_0\gamma^{2i+2}$ for every $i\in\N$, since
$$\frac{r_{i+1}}{\gamma^{2i+2}}=\frac{r_i}{\gamma^{2i}}-\frac{2_s^*-2}{\gamma^{2i+2}}<\frac{r_i}{\gamma^{2i}}
<\cdots<r_0,
$$
and $r_{n+1}>\gamma^{2n+2}$ eventually for $n$ large since
 $\frac{r_{n+1}}{\gamma^{2n+2}}\approx 2^*_s-1>1$, so that
\begin{equation*}
\Big(\prod\limits_{i=0}^{n}r_{i+1}^{\gamma^{2(n-i)}}\Big)^{\frac{1}{r_{n+1}}}
<\Big(\prod\limits_{i=0}^{n}(r_0\gamma^{2i+2})^{\gamma^{2(n-i)}}\Big)^{\frac{1}{\gamma^{2n+2}}}
\leq {r_0}^{\sum\limits_{i=0}^{\infty}\gamma^{-2i-2}}\gamma^{\sum\limits_{i=0}^{\infty}\frac{2i+2}{\gamma^{2i+2}}}<+\infty.
\end{equation*}
Hence $(T_n)$ remains uniformly bounded and the assertion follows.\ Notice that
the $L^\infty$-bound depends on $T_0$ which depends on $u$ (and not only on
$\|u\|_{2^*_s}$) through the presence of $K_0>0$ in estimate \eqref{rt}.
\end{proof}


\section{Poh\v ozaev identity and nonexistence}

\noindent
The purpose of this section is to prove Theorem \ref{teo3}, for this we need the following auxiliary result known as Poh\v ozaev identity for systems
involving the Laplacian fractional operator.

\begin{lemma}
\label{Pozlem}
Let $\Omega$ be a bounded $C^{1,1}$ domain and let
$F\in C^1(\R^+\times \R^+)$ be such that $F_u$ and $F_v$
satisfy the growth conditions \eqref{ff} and \eqref{gg}.
Let $(u,v)\in Y(\Omega)$ be a solution to system
\begin{equation}\label{eq:2.aux}
\left\{
\begin{aligned}
(-\Delta)^s u &= F_{u}(u,v)   && \text{in $\Omega$,}\\
(-\Delta)^s v &= F_v(u,v)   && \text{in $\Omega$,}\\
 u=&\, v=0 && \text{in $\R^N\setminus\Omega$}.
\end{aligned}
\right.
\end{equation}
Then $u,v\in C^s(\R^N)$, $u,v\in C^{1,\alpha}_{{\rm loc}}(\Omega)$ for $s\in (0,1/2)$,
$u,v\in C^{2,\alpha}_{{\rm loc}}(\Omega)$ for $s\in (1/2,1)$ and
\begin{equation}
\label{reg-conc}
\frac{u}{\delta^{s}}\big|_{\Omega},\; \frac{v}{\delta^{s}}\big|_{\Omega}\in C^{\alpha}(\overline{\Omega}) \;\;\mbox{ for some }\alpha\in(0,1),
\end{equation}
where $\delta(x):={\rm dist}(x,\partial\Omega)$, meaning that $u/\delta^{s}|_{\Omega}$ and $ v/\delta^{s}|_{\Omega}$ admit a continuous extension to $\overline{\Omega}$ which is $C^{\alpha}(\overline{\Omega})$. Moreover, the
following identity holds
\begin{align}
\label{idddd}
& \int_{\mathbb{R}^N}\left(|(-\Delta)^{\frac{s}{2}}u|^2+|(-\Delta)^{\frac{s}{2}}v|^2\right)dx-2^*_s\int_{\Omega}F(u,v)dx \\
& \qquad +\frac{\Gamma(1+s)^2}{N-2s}\int_{\partial\Omega}
\Big[\left(\frac{u}{\delta^s}\right)^2+\left(\frac{v}{\delta^s}\right)^2\Big](x,\nu)_{\mathbb{R}^N}d\sigma=0, \notag
\end{align}
where $\Gamma$ is the Gamma function.
\end{lemma}
\begin{proof}
In light of Lemma~\ref{linfty}, we learn that $u,v\in L^\infty(\Omega)$. Then, $F_u(u,v)$ and
$F_v(u,v)$ belong to $L^\infty(\Omega)$ too. In turn, by \cite[Theorem 1.2 and Corollary 1.6]{RS}, we have that $u$ and $v$ satisfy the regularity conclusions stated in \eqref{reg-conc}. 
In particular, the system is satisfied in classical sense.
Whence, we are allowed to apply \cite[Proposition 1.6]{RS1} to both 
components $u$ and $v$, obtaining
\begin{align*}
& \int_{\Omega} (x\cdot\nabla u) (-\Delta)^s u\, dx=\frac{2s-N}{2}\int_{\Omega} u(-\Delta)^su \,dx-\frac{\Gamma(1+s)^2}{2}\int_{\partial\Omega} \Big(\frac{u}{\delta^s}\Big)^2 (x,\nu)_{\mathbb{R}^N}d\sigma,  \\
& \int_{\Omega} (x\cdot\nabla v) (-\Delta)^s v\, dx=\frac{2s-N}{2}\int_{\Omega} v(-\Delta)^sv \,dx-\frac{\Gamma(1+s)^2}{2}\int_{\partial\Omega} \Big(\frac{v}{\delta^s}\Big)^2 (x,\nu)_{\mathbb{R}^N}d\sigma .
\end{align*}
Then, since  $(-\Delta)^s u = F_u(u,v)$ and $(-\Delta)^s v = F_v(u,v)$ weakly
in $\Omega$ and recalling that
$$
\int_{\Omega} u(-\Delta)^su \,dx=\int_{\R^N} |(-\Delta)^{\frac{s}{2}} u|^2 \,dx,
\quad
\int_{\Omega} v(-\Delta)^sv \,dx=\int_{\R^N} |(-\Delta)^{\frac{s}{2}} v|^2 \,dx,
$$
we get
\begin{align*}
& \int_{\Omega} (x\cdot\nabla u) F_u(u,v)\, dx=\frac{2s-N}{2}\int_{\R^N} |(-\Delta)^{s/2} u|^2 \,dx-\frac{\Gamma(1+s)^2}{2}\int_{\partial\Omega} \Big(\frac{u}{\delta^s}\Big)^2 (x,\nu)_{\mathbb{R}^N}d\sigma,  \\
& \int_{\Omega} (x\cdot\nabla v) F_v(u,v)\, dx=\frac{2s-N}{2}\int_{\R^N} |(-\Delta)^{s/2}v|^2 \,dx-\frac{\Gamma(1+s)^2}{2}\int_{\partial\Omega} \Big(\frac{v}{\delta^s}\Big)^2 (x,\nu)_{\mathbb{R}^N}d\sigma .
\end{align*}
Observing that
$
\nabla F(u,v)\cdot x=F_u(u,v)\nabla u\cdot x+F_v(u,v)\nabla v\cdot x,
$
integrating by parts we get,
\begin{align*}
& (2s-N)
\int_{\mathbb{R}^N}\left(|(-\Delta)^{\frac{s}{2}}u|^2+|(-\Delta)^{\frac{s}{2}}v|^2\right)dx+
2N\int_{\Omega}F(u,v)dx \\
& \qquad =\Gamma(1+s)^2\int_{\partial\Omega}
\Big[\left(\frac{u}{\delta^s}\right)^2+\left(\frac{v}{\delta^s}\right)^2\Big](x,\nu)_{\mathbb{R}^N}d\sigma, \notag
\end{align*}
which concludes the proof.
\end{proof}

 \subsection{Proof of nonexistence}
Consider first the case $p+q=2^*_s$ with assumption (1) and assume 
by contradiction that \eqref{eq:1.1+} admits a
positive solution $(u,v)\in Y(\Omega)$. Consider the
functions $f,g:\R^+\times\R^+\to \R$ defined by
$$
f(z,w)=az+bw+\frac{2p}{p+q}z^{p-1}w^q,\quad\,\,
g(z,w)=bz+cw+\frac{2q}{p+q}z^pw^{q-1}.
$$
Then, setting
$$
F(z,w)=\frac{a}{2}z^2+bzw+\frac{c}{2}w^2+\frac{2}{p+q}z^{p}w^q = \frac{1}{2} (A(z,w),(z,w))_{\R^2}+\frac{2}{p+q}z^{p}w^q,
$$
we obtain that $F\in C^1(\mathbb{R}^+\times \mathbb{R}^+)$, $F_z=f$ and $F_w=g$
satisfy the growth conditions \eqref{ff} and \eqref{gg} and $(u,v)$ is a weak solution
to \eqref{eq:1.aux}. Then, the components $u,v$ enjoy the regularity \eqref{reg-conc}
stated in Lemma~\ref{Pozlem} and identity \eqref{idddd} holds.
Testing \eqref{weak} with $(\varphi,\psi)=(u,v)$, yields
\begin{align*}
 \int_{\mathbb{R}^N}\left(|(-\Delta)^{\frac{s}{2}}u|^2+|(-\Delta)^{\frac{s}{2}}v|^2\right)dx
&=\int_{\Omega}f(u,v)u \,dx+ \int_{\Omega}g(u,v)u\, dx\\
& =\int_{\Omega} (AU,U)_{\R^2}dx+ 2\int_{\Omega}u^pv^q dx,
\end{align*}
which substituted in \eqref{idddd}, yields, recalling that $p+q=2^*_s$,
\begin{equation}
\label{equalcase}
 \Big(1-\frac{2^*_s}{2}\Big)\int_{\Omega} (AU,U)_{\R^2}dx
 +\frac{\Gamma(1+s)^2}{N-2s}\int_{\partial\Omega}
\Big[\left(\frac{u}{\delta^s}\right)^2+\left(\frac{v}{\delta^s}\right)^2\Big](x,\nu)_{\mathbb{R}^N}d\sigma=0. \notag
\end{equation}
Since $\Omega$ is star-shaped w.r.t.\ the origin, the equation above yields $\int_{\Omega} (AU,U)_{\R^2}dx\geq 0.$ This is a contradiction
with \eqref{controllo}, because $\mu_2 < 0$ and $u,v>0.$
Now we cover case (2). If $A$ is the zero matrix, we get
$$
\int_{\partial\Omega}
\Big[\left(\frac{u}{\delta^s}\right)^2+\left(\frac{v}{\delta^s}\right)^2\Big](x,\nu)_{\mathbb{R}^N}d\sigma = 0,
$$
which contradicts the fractional version of Hopf lemma, see \cite[Lemma 2.7]{IMS}, since
$(-\Delta)^s u\geq 0$ and $(-\Delta)^s v\geq 0$ weakly yield $\frac{u}{\delta^s}\geq \omega$
and $\frac{v}{\delta^s}\geq \omega'$, for some positive constants $\omega,\omega'$.
Let us turn to case (3). If $\varphi_1>0$ is the first eigenfunction corresponding to $\lambda_{1,s}$ and we assume that a solution of \eqref{eq:1.1+} exists, by choosing $(\varphi_1,0)$
and $(0,\varphi_1)$ respectively in \eqref{weak}, we get
\begin{align*}
& \lambda_{1,s}\int_{\Omega} u\varphi_1 dx=
\int_{\R^N}(-\Delta)^{\frac{s}{2}}u  (-\Delta)^{\frac{s}{2}} \varphi_1dx=
\int_{\Omega} \big(au\varphi_1+bv\varphi_1+\frac{2p}{p+q}u^{p-1} v^q\varphi_1\big)dx, \\
& \lambda_{1,s}\int_{\Omega} v\varphi_1 dx=
\int_{\R^N}(-\Delta)^{\frac{s}{2}}v  (-\Delta)^{\frac{s}{2}} \varphi_1dx=
\int_{\Omega} \big(bu\varphi_1+cv\varphi_1+\frac{2q}{p+q}u^{p} v^{q-1}\varphi_1\big)dx.
\end{align*}
Then, since $b\geq 0$ and $u,v>0$, we get
$$
\lambda_{1,s}\int_{\Omega} u\varphi_1 dx>a\int_{\Omega} u\varphi_1 dx,\qquad 
\lambda_{1,s}\int_{\Omega} v\varphi_1 dx>c\int_{\Omega} v\varphi_1 dx,
$$
that is $\max\{a,c\}<\lambda_{1,s}$. On the other hand, by assumption and a direct calculation
$$
\lambda_{1,s}-|a-c|\leq \mu_1=\frac{(a+c)-\sqrt{(a-c)^2+4b^2}}{2}\leq \frac{(a+c)-|a-c|}{2}=\min\{a,c\},
$$
which yields $\max\{a,c\}\geq \lambda_{1,s}$, namely a contradiction. 
Finally we prove the last assertion. In the case $p+q>2^*_s$,  any bounded solution
of system \eqref{eq:1.1+} is smooth according to Lemma~\ref{Pozlem} and 
arguing as above yields the identity
\begin{align*}
\Big(1-\frac{2^*_s}{2}\Big)\int_{\Omega} (AU,U)_{\R^2}dx
&+2\Big(1-\frac{2^*_s}{p+q}\Big) \int_{\Omega} u^p v^q dx \\
&+\frac{\Gamma(1+s)^2}{N-2s}\int_{\partial\Omega}
\Big[\left(\frac{u}{\delta^s}\right)^2+\left(\frac{v}{\delta^s}\right)^2\Big](x,\nu)_{\mathbb{R}^N}d\sigma=0.
\end{align*}
This yields $\int_{\Omega} (AU,U)_{\R^2}dx>0$, contradicting $\mu_2\leq 0$ via
\eqref{controllo}. This concludes the proof.
\qed
\medskip

\noindent
{\bf Proof of Theorem \ref{regul}.}
The assertion follows as a particular case of Lemma~\ref{Pozlem}.
\qed
\medskip

\section{Existence I, subcritical case}

\noindent
In this section, we will prove the Theorem \ref{teo1} which guarantees the existence of 
solutions for the problem \ref{eq:1.1+} involving subcritical non-linearity.

 \subsection{Proof of existence I}

Let $\Omega $ be a bounded domain and suppose that
\begin{equation}\label{d1}
b \geq 0,
\end{equation}
\begin{equation}\label{d2}
\mu_2 < \lambda_{1,s},
\end{equation}
\begin{equation}\label{d3}
p + q < 2^*_s.
\end{equation}

\noindent
Consider the functional $ I: Y(\Omega) \to \mathbb{R}$ defined by
$$ 
I(U) := \frac{1}{2}   \| U \|^2_Y -
\frac{1}{2} \int_{\Omega} ( AU,U )_{\R^2} dx. 
$$
We shall minimize the functional $I$ restricted to the set
$$
\M:= \left\{ U=(u,v) \in Y(\Omega) \; : \; \int_{\Omega} (u^+)^p (v^+)^q dx =1 \right\}.
$$
By virtue of \eqref{d2} the embedding $X(\Omega) \hookrightarrow L^2(\Omega)$
(with the sharp constant $\lambda_{1,s}$), we have
\begin{equation}\label{d4}
 I(U) \geq   \frac{1}{2} \min \left\{ 1, \left( 1 - \frac{\mu_2}{\lambda_{1,s}}\right) \right\}
 \left\|U
\right\|^2_Y \geq 0.
\end{equation}
So define
\begin{equation}
I_0:=\inf_{\M} I ,
\end{equation}
and let $(U_n)= (u_n, v_n) \subset \M$ be a minimizing sequence for $I_0$.
Then $I(U_n)= I_0 + o_n(1) \leq C,$ for some $C>0$ (where $o_n(1) \rightarrow 0,$ as $n \rightarrow \infty$) and consequently by \eqref{d4}, we get
\begin{equation}\label{d5}
[u_n]_s^2 + [v_n]_s^2 =\|u_n\|_X^2+\|v_n\|_X^2=\|U_n\|^2_Y \leq C'.
\end{equation}
Hence, there are two subsequences of $(u_n)\subset X(\Omega)$ and $(v_n)\subset X(\Omega)$ (that we will still label as $u_n$ and $v_n$)
such that $U_n = (u_n,v_n)$ converges to some $U=(u,v)$ in $Y(\Omega)$ weakly and
\begin{align}
\label{weakcc}
& [u]_s^2 \leq \liminf_{n} \frac{C(N,s)}{2}\int_{\mathbb{R}^{2N}} \frac{|u_n(x)-u_n(y)|^2}{|x-y|^{N+2s}} dx dy, \\
\label{weakcc1}
& [v]_s^2 \leq  \liminf_{n} \frac{C(N,s)}{2}\int_{\mathbb{R}^{2N}} \frac{|v_n(x)-v_n(y)|^2}{|x-y|^{N+2s}} dx dy.
\end{align}
Furthermore, in view of the compact embedding $X(\Omega) \hookrightarrow L^\sigma(\Omega)$ for all $\sigma<2^*_s$ (cf.\ \cite[Corollary 7.2]{nezza}), we get that $U_n = (u_n,v_n)$ converges to $(u,v)$ strongly in $(L^{p+q}(\Omega))^2$, as $n\to\infty$.
Of course, up to a further subsequence, we have that $(u_n(x),v_n(x))$ converges to $(u(x),v(x))$ for a.e.\ $x\in\R^N$. 
Now we will show that  $U:=(u,v) \in \M.$
Indeed, since $(U_n) \subset \M$, we have 
\begin{equation}
\label{constr}
\int_{\Omega} (u_n^+)^p (v_n^+)^q dx=1.
\end{equation}
Since
$$
\lim_{n} \int_{\Omega} |u_n|^{p+q}dx = \int_{\Omega} |u|^{p+q}dx,\qquad
\lim_{n}\int_{\Omega} |v_n|^{p+q}dx = 
\int_{\Omega} |v|^{p+q}dx,
$$
we have in particular $|u_n|^{p+q}\leq \eta_1$ and $|v_n|^{p+q}\leq \eta_2$, for 
some $\eta_i \in L^1(\Omega)$ and any $n\in\N$. Then
$$
(u_n^+)^p (x) (v_n^+)^q (x) \leq \frac{p}{p+q}|u_n(x)|^{p+q} + \frac{q}{p+q}|v_n (x)|^{p+q}\leq \eta_1(x)+\eta_2(x),
\quad \mbox{for a.e. in $\Omega$}.
$$
In turn, by the Dominated Convergence Theorem, 
passing to the limit in \eqref{constr}, we obtain
$$
\int_{\Omega} (u^+)^p (v^+)^q dx=1,
$$
and, consequently $U=(u,v) \in \M$ with $u,v \neq 0.$
We now show that $U=(u,v)$ is, indeed, a minimizer for $I$ on $\M$ and both the components $u,v$ are nonnegative.
By passing to the limit in $I(U_n)= I_0 + o_n(1)$, where $o_n(1)\rightarrow 0$ as $n \rightarrow \infty,$ using \eqref{weakcc} and \eqref{weakcc1} and the strong
convergence of $(u_n,v_n)$ to $(u,v)$ in $(L^{2}(\Omega))^2$, as $n\to\infty$,
we conclude that $I(U) \leq I_0$.
Moreover, since $U \in \M$ and $ I_0 = \inf_{\M} I \leq I(U)$, we achieve that $ I(U) = I_0.$  This proves the minimality of $U \in \M.$
On the other hand, let
$$
G(U)= \int_{\Omega} (u^{+})^p (v^{+})^q dx -1,
$$
where $U=(u,v) \in Y(\Omega).$
Note that $G \in C^1$ and since $U \in \M$,
$$
G'(U)U= (p+q) \displaystyle\int_{\Omega}  (u^{+})^p  (v^{+})^q dx = p+q \neq 0,
$$ 
hence, by Lagrange Multiplier Theorem, there exists a multiplier $\mu \in \mathbb{R}$ such that
\begin{equation}\label{d6}
I'(U)(\varphi,\psi)=\mu G'(U)(\varphi,\psi) , \;\;\;\; \forall (\varphi,\psi) \in Y(\Omega).
\end{equation}
Taking $(\varphi,\psi)= (u^-, v^-):=U^-$ in \eqref{d6}, we get
\begin{align*}
\|U^-\|^2_Y &=\frac{C(N,s)}{2}\int_{\R^{2N}}\frac{u^+(x)u^-(y)+u^-(x)u^+(y)}{|x-y|^{N+2s}} dxdy  \\
&+\frac{C(N,s)}{2}\int_{\R^{2N}}\frac{v^+(x)v^-(y)+v^-(x)v^+(y)}{|x-y|^{N+2s}} dxdy \\
&  +\int_{\Omega} ( AU,U^-)_{\R^2} dx.
\end{align*}
Dropping this formula into the expression of $I(U^-)$, we have 
\begin{align}
\label{d7}
 I(U^-)& = \frac{b}{2} \int_{\Omega} ( v^+ u^- + u^+ v^-) dx +\frac{C(N,s)}{4}\int_{\R^{2N}}\frac{u^+(x)u^-(y)+u^-(x)u^+(y)}{|x-y|^{N+2s}} dxdy\\
&+\frac{C(N,s)}{4}\int_{\R^{2N}}\frac{v^+(x)v^-(y)+v^-(x)v^+(y)}{|x-y|^{N+2s}} dxdy\leq 0, \notag
\end{align}
since $b\geq0$, $w^-\leq 0$ and $w^+\geq 0$. Furthermore, 
$$
I(U^-) \geq    \min \left\{ 1, \left( 1 - \frac{\mu_2}{\lambda_{1}}\right) \right\}
 \left\|U^- \right\|^2_Y \geq 0
 $$ 
 and using \eqref{d7}, we get $U^- = (u^-, v^-)=(0,0)$ and therefore 
 $u,v\geq 0.$ We now prove the existence of a positive solution to \eqref{eq:1.1}.
Using again \eqref{d6}, we see that
\begin{equation}\nonumber
\|U\|^2_Y - \displaystyle\int_{\Omega} (AU,U)_{\R^2}dx - \mu (p+q) \displaystyle\int_{\Omega}  u^p v^q dx = 0
\end{equation}
and since $U \in \M$, we conclude that
\begin{equation*}
I_0=I(U)= \frac{ \mu (p+q)}{2}>0,
\end{equation*}
since $I_0$ is positive, via \eqref{d2}. Then, by \eqref{d6},
$U$ satisfies the following system, weakly,
\begin{equation*}
\left\{
\begin{array}{lc}
(-\Delta)^s u=au+bv+\dfrac{2p I_0}{p+q } {u}^{p
-1}{v}^q , &  \; \Omega \\
(-\Delta)^s v=bu+cv+\dfrac{2q I_0}{p+q } {u}^p
{v}^{q-1}, & \; \Omega \\
u=v=0,  & \; \mathbb{R}^N \backslash \Omega.
\end{array}
\right.
\end{equation*}
Now using the homogeneity of system, we 
get $\tau>0$ such that $W = (I_0)^{\tau} U$ is a solution of \eqref{eq:1.1+}. Since $b\geq 0$
and $u,v\geq 0$ we get, in weak sense
\begin{equation*}
\left\{
\begin{array}{lc}
(-\Delta)^s u\geq au &  \; \Omega \\
(-\Delta)^s v\geq cv & \; \Omega \\
u\geq 0,\,\, v\geq 0 & \; \Omega \\
u=v=0,  & \; \mathbb{R}^N \backslash \Omega.
\end{array}
\right.
\end{equation*}
By the strong maximum principle (cf.\ \cite[Theorem 2.5]{IMS}), 
we conclude $u,v>0$ in $\Omega$.
\qed

\section{Existence II, critical case}

\noindent
Next we turn to Theorem \ref{teo2}, for the critical case $p+q=2^*_s$. 
The variational tool used is the Mountain Pass Theorem.
The embedding $X(\Omega)\hookrightarrow  L^{2^*_s}(\Omega)$ is {\em not} 
compact, but we will show that, below a certain level $c$, the associated functional satisfies the Palais-Smale condition.

\subsection{Preliminary results}
We will make use of the following definition
\begin{equation}\label{Ss}
\S_s:=\inf\limits_{u \in X(\Omega)\setminus\{0\}}\S_s (u),
\end{equation}
where
\begin{equation}\label{quoSs}
\S_s (u):=\frac{\displaystyle\int_{\R^{N}}|(-\Delta)^{\frac{s}{2}} u|^2dx}{\Big(\displaystyle\int_{\R^{N}}|u(x)|^{2^{*}_{s}}dx \Big)^{\frac{2}{2^{*}_{s} } }}
\end{equation}
is the associated Rayleigh quotient.
We also define the following related minimizing problems
\begin{equation}\label{quoSpq}
\S_{p+q} (\Omega):=\inf\limits_{u \in X(\Omega)\setminus\{0\}}
\frac{\displaystyle\int_{\R^{N}}|(-\Delta)^{\frac{s}{2}} u|^2dx}
{\Big(\displaystyle\int_{\R^{N}}|u(x)|^{p+q}dx \Big)^{\frac{2}{p+q} }}
\end{equation}
and
\begin{equation}\label{quoSpqtil}
\widetilde{\S}_{p,q} (\Omega):=\inf\limits_{u,v \ \in X(\Omega)\setminus\{0\}} \frac{\displaystyle\int_{\R^{N}}\left(|(-\Delta)^{\frac{s}{2}} u|^2+|(-\Delta)^{\frac{s}{2}} v|^2\right)dx}
{\Big(\displaystyle\int_{\R^{N}}|u(x)|^{p}|v(x)|^{q}dx \Big)^{\frac{2}{p+q} }}.
\end{equation}
We shall also agree that
 $$
 \S_s =\S_{p+q}(\Omega), \quad \widetilde{\S}_{s}:=\widetilde{\S}_{p,q} (\Omega),
 \qquad \text{if $p+q =2^{*}_{s}$}.
 $$
The following result, in the local case, was proved in \cite{alves}. The proof follows 
by arguing as it was made in \cite{alves}, but, for the sake of 
completeness, we present its proof.

\begin{lemma}\label{relation}
	Let $\Omega $ be a domain, not necessarily bounded, and $p+q\leq 2^{*}_s.$ Then
\begin{equation}\label{iq3}
\widetilde{\S}_{p,q}(\Omega )=\left[ \left( \frac pq
\right) ^{\frac q {p+q}}+\left( \frac pq \right) ^{%
\frac{-p }{p+q }}\right] \S_{p+q }(\Omega).
\end{equation}
Moreover, if $w_0$ realizes $\S_{p+q }(\Omega )$ then $(B w_0,C w_0)$
realizes $\widetilde{\S}_{p,q}(\Omega )$, for all positive constants $B$ and
$C$ such that $B/C=\sqrt{p/q}.$
\end{lemma}
\begin{proof}
 Let $\{w_n\}\subset X(\Omega)\setminus\{0\}$ be a minimizing sequence for $\S_{p+q}(\Omega).$ Define $u_n:=s w_n$ and $v_n:=t w_n$, where $s,t>0$ 
 will be chosen later on. By definition \eqref{quoSpqtil}, we get
\begin{equation}\label{eq2}
\frac{g\left(\displaystyle\frac{s}{t}\right)\displaystyle\int_{\R^{N}}|(-\Delta)^{\frac{s}{2}} w_n|^2dx}
{\Big(\displaystyle\int_{\R^{N}}|w_n(x)|^{p+q}dx \Big)^{\frac{2}{p+q} }}\geq \widetilde{\S}_{p,q} (\Omega),
\end{equation}
where $g:\R^+\to\R^+$ is defined by 
$$
g(x):=x^{2q/p+q}+x^{-2p/p+q},\quad x>0.
$$  
The minimum value is assumed by $g$ at the point $x=\sqrt{p/q},$ and it is given by
$$
g(\sqrt{p/q})=\left(\Big(\frac{p}{q}\Big)^{\frac{q}{p+q}}
+\Big(\frac{p}{q}\Big)^{-\frac{p}{p+q}}\right).
$$
Whence, by choosing $s,t$ in \eqref{eq2} so that
 $s/t=\sqrt{p/q}$, and passing to the limit, we obtain
$$
\widetilde{\S}_{p,q} (\Omega)\leq g(\sqrt{p/q})\S_{p+q}(\Omega).
$$
In order to prove the reverse inequality, let $\{(u_n, v_n)\}\subset 
(X(\Omega)\setminus\{0\})^2$ be a minimizing sequence for $\widetilde{\S}_{p,q} (\Omega)$ and define
$z_n:=s_n v_n$ for some $s_n >0$ such that 
$$
\int_{\R^N} |u_n|^{p+q}dx= \int_{\R^N} |z_n|^{p+q}dx.
$$
Then, by Young's inequality, we obtain
\begin{align}\label{eq3}
	\int_{\R^N}|u_n|^{p} |z_n|^{q}dx&\leq  \frac{p}{p+q}\int_{\R^N} |u_n|^{p+q}dx+  \frac{q}{p+q}\int_{\R^N} |z_n|^{p+q}dx \nonumber\\
&= \int_{\R^N} |u_n|^{p+q}dx= \int_{\R^N} |z_n|^{p+q}dx.
\end{align}
Thus, using \eqref{eq3}, we obtain
\begin{align*}
\lefteqn{\frac{\displaystyle\int_{\R^{N}}\left(|(-\Delta)^{\frac{s}{2}} u_n|^2+|(-\Delta)^{\frac{s}{2}} v_n|^2\right)dx}{\Big(\displaystyle\int_{\R^{N}}|u_n(x)|^{p}|v_n(x)|^{q}dx \Big)^{\frac{2}{p+q} }} }\\
&=  s^{2q/p+q}_{n} \frac{\displaystyle\int_{\R^{N}}\left(|(-\Delta)^{\frac{s}{2}} u_n|^2+|(-\Delta)^{\frac{s}{2}} v_n|^2\right)dx}{\Big(\displaystyle\int_{\R^{N}}|u_n(x)|^{p}|z_n(x)|^{q}dx \Big)^{\frac{2}{p+q} }}\\
&\geq  s^{2q/p+q}_{n}\frac{\displaystyle\int_{\R^{N}}|(-\Delta)^{\frac{s}{2}} u_n|^2dx}{\Big(\displaystyle\int_{\R^{N}}|u_n(x)|^{p+q}dx \Big)^{\frac{2}{p+q} }}
 +  s^{-2p/p+q}_{n} \frac{\displaystyle\int_{\R^{N}}|(-\Delta)^{\frac{s}{2}} z_n|^2dx}{\Big(\displaystyle\int_{\R^{N}}|z_n(x)|^{p+q}dx \Big)^{\frac{2}{p+q} }}\\
&\geq  g(s_n) \S_{p+q}(\Omega)\geq g(\sqrt{p/q}) \S_{p+q}(\Omega).
\end{align*}
Therefore, letting $n\to\infty$
in the above inequality, we get the reverse inequality, as desired.
From \eqref{iq3}, the last assertion immediately follows and
the proof is concluded.
\end{proof}

\vskip3pt
\noindent
From \cite[Theorem 1.1]{coti}, we learn that $\S_s$ is attained. Precisely $\S_s=\S_s(\widetilde{u}),$ where
\begin{equation}\label{xo} \widetilde{u}(x)=\frac{k}{( \mu^2 +|x-x_0|^2)^{\frac{N-2s}{2}}}, 
\quad x \in \R^N, \ k \in \R\setminus\{0\}, \,\, \mu>0, \,\, x_0 \in\R^N.
 \end{equation}
Equivalently,
$$
\S_s=\inf\limits_{\begin{tiny} \begin{array}{cc}
                        u \in X(\Omega)\setminus\{0\} \\
                        \|u\|_{L^{2^{*}_s}}=1\\
                      \end{array} \end{tiny}}  \int_{\R^{N}}|(-\Delta)^{\frac{s}{2}} u|^2dx=\int_{\R^{N}}|(-\Delta)^{\frac{s}{2}} \overline u|^2dx,
                    $$
where $\overline{u}(x)=\widetilde{u}(x)/\|\widetilde{u}\|_{L^{2^{*}_s}}.$
In what follows, we suppose that, up to a translation, $x_0 = 0$ in \eqref{xo}.
The function
$$
u^*(x):=
\overline{u}\Big(\frac{x}{{\mathcal S}_s^{\frac{1}{2s}}}\Big),\, \quad x \in \R^N,
$$
is a solution to the problem
\begin{equation}\label{4.6}
(-\Delta)^{s} u= |u|^{2^{*}_s-2}u \quad \text{in $\R^N$},
\end{equation}
verifying the property
\begin{equation}\label{4.7}
\|u^*\|^{2^{*}_s}_{L^{2^{*}_s}(\R^N)}=\S^{N/2s}_{s}.
\end{equation}
Define the family of functions
$$
U_\epsilon(x)=\epsilon^{-\frac{N-2s}{2}}u^* \Big(\frac{x}{\epsilon} \Big), \quad x \in \R^N,
$$
\noindent
then $U_\epsilon$ is a solution of (\ref{4.6}) and verifies, for all $\epsilon >0,$
\begin{equation}\label{4.8}
\int_{\R^{N}}|(-\Delta)^{\frac{s}{2}} U_\epsilon|^2dx=\int_{\R^N} |U_\epsilon (x)|^{2^{*}_s} dx=\S^{N/2s}_{s}.
\end{equation}
Fix $\delta >0$ such that $B_{4\delta} \subset \Omega$ and $\eta \in C^{\infty}(\R^N)$ a cut-off
function such that $0\leq \eta \leq 1$ in $\R^N,$ $\eta=1$ in $B_{\delta}$ and $ \eta=0$ in $B^{c}_{2\delta}=\R^N \setminus B_{2\delta},$ where $B_r=B_r (0)$ is the ball centered at origin with radius $r>0$.
Now define the family of nonnegative truncated functions
\begin{equation}\label{uepsilon}
  u_\epsilon(x)=\eta(x) U_\epsilon(x), \ x \in \R^N,
\end{equation}
and note that $u_\epsilon \in X(\Omega).$
%
%
The following result was proved in \cite{servadeiTAMS}
and it constitutes the natural fractional counterpart of those
proved for the local case in \cite{bn}.

\begin{proposition}\label{prop2} Let $s \in (0,1)$ and $N>2s$. Then the following facts hold.
\begin{itemize}
\item[a)] $$\int_{\R^{N}}|(-\Delta)^{\frac{s}{2}} u_\epsilon|^2dx \leq \S^{N/2s}_{s}+ \O(\epsilon^{N-2s}), \quad\text{as $\epsilon \to 0$}. $$
\item [b)]$$
\int_{\R^N} |u_\epsilon(x)|^{2} dx \geq
\left\{\begin{array}{lcr} C_s \epsilon^{2s} + \O(\epsilon^{N-2s}) &\mbox{if}&N > 4s,\\
 C_s \epsilon^{2s}|\log \epsilon| + \O(\epsilon^{2s}) &\mbox{if}&N = 4s,\\
 C_s \epsilon^{N-2s} + \O(\epsilon^{2s}) &\mbox{if}&2s< N< 4s,\\
\end{array}\right.
$$  
as  $\epsilon \to 0$. Here $C_s$ is a positive constant depending only on $s.$
\item [c)]
$$
\int_{\R^N} |u_\epsilon(x)|^{2^{*}_s} dx=\S^{N/2s}_{s}+ \O(\epsilon^{N}),
\quad\text{as $\epsilon \to 0$}. 
$$
\end{itemize}\end{proposition}

\noindent
Consider now, for any $\lambda\geq 0$, the following minimization problem
$$
\S_{s,\lambda}:=\inf_{v\in X(\Omega)\setminus\{0\}}\S_{s,\lambda}(v),
$$
where
$$
\S_{s,\lambda}(v):=\frac{\displaystyle\int_{\R^{N}}|(-\Delta)^{\frac{s}{2}} v|^2dx- \lambda \displaystyle\int_{\R^N}|v(x)|^2 dx}{\Big(\displaystyle\int_{\R^{N}}|v(x)|^{2^{*}_{s}}dx \Big)^{\frac{2}{2^{*}_{s} } }}.
$$
The following result was proved in \cite[Propositions 21 and 22]{servadeiTAMS}
for the first assertion, and in \cite[Corollary 8]{low} for the second assertion.

\begin{proposition}\label{scalar}
	Let $s \in (0,1)$ and $N>2s$. Then the following facts hold.
\begin{itemize}
\item [a)] For $N \geq 4s,$
$$   
\S_{s,\lambda}(u_{\epsilon})< \S_s, \quad \text{for all $\lambda >0$ and any
	 $\epsilon>0$ sufficiently small.}
$$
\item [b)] For $2s < N < 4s$, there exists $\lambda_s >0$ such that 
 $$   \S_{s,\lambda}(u_{\epsilon})< \S_s, 
 \quad \text{for all $\lambda >\lambda_s$ and any
 	$\epsilon>0$ sufficiently small.}
 $$
\end{itemize}
\end{proposition}

\begin{proof} For the sake of the completeness, we sketch the proof. 
	\vskip4pt
\noindent{\sc Case:} $N> 4s.$ By Proposition \ref{prop2}, we infer 
\begin{align*}
 \S_{s,\lambda}(u_{\epsilon}) &\leq \frac{ \S^{N/2s}_{s}+ \O(\epsilon^{N-2s})-\lambda C_s\epsilon^{2s}}{\Big( \S^{N/2s}_{s}+ \O(\epsilon^{N})\Big)^{\frac{2}{2^{*}_{s} } }}\\
&\leq \S_{s}+ \O(\epsilon^{N-2s})-\lambda \widetilde{C}_s\epsilon^{2s}, \\
&\leq \S_{s}+ \epsilon^{2s}(\O(\epsilon^{N-4s})-\lambda \widetilde{C}_s)\\
&< \S_s, \quad \text{for all $\lambda>0$ and $\epsilon>0$ small enough
	and some $\widetilde{C}_s>0$}.
\end{align*}

\noindent{\sc Case:} $N=4s.$

\begin{align*}
 \S_{s,\lambda}(u_{\epsilon})&\leq \frac{ \S^{N/2s}_{s}+ \O(\epsilon^{N-2s})-\lambda C_s\epsilon^{2s}|\log \epsilon|+ \O(\epsilon^{2s})}{\Big( \S^{N/2s}_{s}+ \O(\epsilon^{N})\Big)^{\frac{2}{2^{*}_{s} } }}\\
&\leq \S_{s}+ \O(\epsilon^{2s})-\lambda \widetilde{C}_s\epsilon^{2s}|\log \epsilon|,\\
&\leq \S_{s}+ \epsilon^{2s}(\O(1)-\lambda \widetilde{C}_s|\log \epsilon|)\\
&< \S_s, \quad \text{for all $\lambda>0$ and $\epsilon>0$ small enough
	and some  $\widetilde{C}_s>0$}.
\end{align*}

\noindent{\sc Case:} $2s<N<4s.$

\begin{align*}
 \S_{s,\lambda}(u_{\epsilon})&\leq\frac{ \S^{N/2s}_{s}+ \O(\epsilon^{N-2s})-\lambda C_s\epsilon^{N-2s}+\O(\epsilon^{2s})}{\Big( \S^{N/2s}_{s}+ \O(\epsilon^{N})\Big)^{\frac{2}{2^{*}_{s} } }}\\
&\leq \S_{s}+ \epsilon^{N-2s}(\O(1)-\lambda \widetilde{C}_s) + \O(\epsilon^{2s}), \\
&< \S_s, \quad \text{for all $\lambda>0$ large enough ($\lambda \geq \lambda_s$),  
	$\epsilon$ sufficiently small and some $\widetilde{C}_s>0$}.
\end{align*}
This concludes the sketch.
\end{proof}

\noindent 
Even if it is not strictly necessary for the proof of our main result, we state
the following Corollary for possible future usage.

\begin{corollary}
	\label{prop4} 
	Suppose that $\mu_1$ given in $(\ref{controllo})$ is positive and let
$$
\widetilde{\S}_{s,A}=\inf_{u,v\ \in X(\Omega)\setminus\{0\}}\S_{s,A}(u,v),
$$
where
$$
\S_{s,A}(u,v)=\frac{\displaystyle\int_{\R^{N}}\left(|(-\Delta)^{\frac{s}{2}} u|^2+|(-\Delta)^{\frac{s}{2}} v|^2\right)dx-
		\displaystyle \int_{\R^N}\big(A(u(x),v(x)),(u(x),v(x))\big)_{\R^2} dx}{\Big(\displaystyle \int_{\R^{N}}|u(x)|^p|v(x)|^{q}dx \Big)^{\frac{2}{2^{*}_{s} } }}
$$ 
where $ p+q= 2^{*}_{s}.$ Then the following facts holds
\begin{itemize}
\item [a)] If $N \geq 4s,$ then
$$   \widetilde{\S}_{s,A}< \widetilde{\S}_{s}.$$
\item [b)] For $2s < N < 4s,$ there exists $\mu_{1,s} >0,$ such that if $\mu_1 > \mu_{1,s},$ we have
 $$ 
 \widetilde{\S}_{s,A}< \widetilde{\S}_{s}.
 $$
\end{itemize}
\end{corollary}
\begin{proof}
From Proposition \ref{scalar}, we have
\begin{itemize}
\item [a)] For $N \geq 4s,$ we have
$$   
\S_{s,\mu_1}(u_{\epsilon})< \S_s, \quad \text{if $\mu_1 >0$
	and provided $\epsilon >0$ is sufficiently small.}
$$
\item [b)] For $2s < N < 4s,$ there exists $\mu_{1,s }>0,$ such that if $\mu_1 > \mu_{1,s},$ we have
 $$ 
 \S_{s,\mu_1}(u_{\epsilon})< \S_s, \quad \text{provided $\epsilon >0$ is sufficiently small.}
 $$
\end{itemize}
Let $B,C>0$ be such that $\frac{B}{C}=\sqrt{\frac{p}{q}}.$
From \eqref{controllo} and the above inequalities, we infer that
\begin{align*}
\widetilde{\S}_{s,A}    &\leq \S_{s,A}(Bu_{\epsilon}, C u_{\epsilon})\\
&\leq \frac{(B^2 +C^2) \Big(\displaystyle \int_{\R^{N}}|(-\Delta)^{\frac{s}{2}} u_\epsilon|^2dx -
\mu_1  \displaystyle\int_{\R^N}|u_{\epsilon}(x)|^2 dx \Big) }{(B^p C^q)^{2/2^{*}_{s} } \Big(\displaystyle\int_{\R^N}|u_{\epsilon}(x)|^{2^{*}_{s}}dx \Big)^{2/2^{*}_{s}}}\\
&=\Big[\Big(\frac{p}{q}\Big)^{q/p+q}+\Big(\frac{p}{q}\Big)^{-p/p+q}\Big]\S_{s,\mu_1}(u_{\epsilon})\\
&< \Big[\Big(\frac{p}{q}\Big)^{q/p+q}+\Big(\frac{p}{q}\Big)^{-p/p+q}\Big]\S_{s}= \widetilde{\S}_{s}.
\end{align*}
This concludes the proof.
\end{proof}

\subsection{Proof of existence II}

In order to get weak solutions to system \eqref{eq:1.1+}, 
we now define the functional  $J:Y(\Omega)\rightarrow \R$ by setting
\begin{align*}
J(u,v)&=\frac{1}{2}\int_{\R^{N}}\left(|(-\Delta)^{\frac{s}{2}} u|^2+|(-\Delta)^{\frac{s}{2}} v|^2\right)dx\\
&- \frac{1}{2}\int_{\R^N}(A(u,v),(u,v))_{\R^2} dx-\frac{2}{ 2^{*}_{s}} \int_{\R^N}(u^+)^p (v^+)^q dx,
\end{align*}
whose Gateaux derivative is given by 
\begin{align}
\label{deriv}
& J'(u,v)(\phi, \psi)  \\ &=\frac{C(N,s)}{2}\int_{\R^{2N}}\frac{(u(x)-u(y))(\phi(x)-\phi(y))+(v(x)-v(y))(\psi(x)-\psi(y))}{|x-y|^{N+2s}} dxdy \notag \\
&-\int_{\Omega}(A(u,v),(\phi,\psi))_{\R^2} dx-\frac{2p}{ 2^{*}_{s}} \int_{\Omega} (u^+)^{p -1}(v^+)^q \phi \,dx
-\frac{2q}{ 2^{*}_{s}} \int_{\Omega} (v^+)^{q -1}(u^+)^p \psi \, dx, \notag
\end{align}
for every $(\phi,\psi)\in Y(\Omega)$.
We shall observe that the weak solutions of problem \eqref{eq:1.1+} correspond to the critical 
points of the functional $J.$
%
Under hypothesis $0<\mu_1 \leq \mu_2 <\lambda_{1,s}$, our goal 
is to prove Theorem \ref{teo2}.
We first show that $J$ satisfies the Mountain Pass Geometry. 

\begin{proposition}\label{MPG} Suppose $\mu_2< \lambda_{1,s}.$ The functional $J$ satisfies
\begin{itemize}
\item[a)] There exist $\beta, \rho>0$ such that $J(u,v)\geq \beta$ if $\|(u,v)\|_{Y}=\rho$;
\item[b)] there exists $(e_1,e_2) \in Y(\Omega) \backslash\{(0,0)\}$ with $\|(e_1,e_2)\|_{Y}>\rho$ such that $J(e_1,e_2)\leq 0$.
\end{itemize}
\end{proposition}

\begin{proof}
	$a)$ By means of \eqref{controllo}, using 
$$
(u^+)^p (v^+)^q \leq |u|^{p+q} + |v|^{p+q}= |u|^{2^*_s} + |v|^{2^*_s} 
$$ 
and Poincar\'e inequality, we have
$$
J(u,v)\geq\frac{1}{2} \Big( 1 - \frac{\mu_2}{\lambda_{1,s}} \Big)\|(u,v)\|^{2}_{Y}- C\|(u,v)\|^{2^{*}_s}_{Y},
$$
where $C>0$ is a constant. 

\noindent
$b)$ Choose $(\tilde u_0,\tilde v_0)\in Y(\Omega)\setminus\{(0,0)\}$ with
$\tilde u_0\geq 0$, $\tilde v_0\geq 0$ a.e.\ and $\tilde u_0\tilde v_0\neq0$. Then
\begin{align*}
J(t\tilde u_0,t\tilde v_0)&=\frac{t^2}{2}\int_{\R^{N}}\left(|(-\Delta)^{\frac{s}{2}} \tilde u_0|^2+|(-\Delta)^{\frac{s}{2}} \tilde v_0|^2\right)dx\\
&- \frac{t^2}{2}\int_{\R^N}(A(\tilde u_0,\tilde v_0),(\tilde u_0,\tilde v_0)) dx-\frac{2t^{2^{*}_s}}{ 2^{*}_{s}} \int_{\R^N}\tilde u_0^p \tilde v_0^q dx,
\end{align*}
by choosing $t>0$ sufficiently large, the assertion follows.
This concludes the proof.
\end{proof}

\noindent
Therefore, by the previous facts, by the Mountain Pass Theorem it follows
that there exists a sequence $\{(u_n,v_n)\}\subset
Y(\Omega)$, so called {\it  $(PS)_c$--Palais Smale sequence at level $c$}  , such that
\begin{equation}
\label{PS}
J(u_n,v_n)\rightarrow  c, \quad  \|J'(u_n,v_n)\|\rightarrow 0,
\end{equation}
where $c$  is given by
$$
c=\inf_{\gamma \in \Gamma}\max_{t\in[0,1]}J(\gamma(t)),
$$
with
$$
\Gamma=\{\gamma \in C([0,1],Y(\Omega)): \gamma(0)=(0,0)\ \mbox{and}\ J(\gamma(1))\leq 0\}.
$$
\noindent
Next we turn to the boundedness of $\{(u_n,v_n)\}$ in $Y(\Omega)$.
\begin{lemma}[Boundedness]
\label{bounded}
The $(PS)$ sequence $\{(u_n,v_n)\} \subset Y(\Omega)$ is bounded.
\end{lemma}

\begin{proof}
We have for every $n\in\N$
\begin{align*}
C +C\|(u_n,v_n)\|_{Y}&\geq J(u_n,v_n)-\frac{1}{2^{*}_s}J'(u_n,v_n)(u_n,v_n)\\
&= \Big(\frac{1}{2}-\frac{1}{2^{*}_s}\Big)\|(u_n,v_n)\|^{2}_{Y} -\Big(\frac{1}{2}-\frac{1}{2^{*}_s}\Big)\int_{\R^N}\big(A(u_n,v_n),(u_n,v_n)\big)_{\R^2} dx\\
&\geq \Big(\frac{1}{2}-\frac{1}{2^{*}_s}\Big) \Big( 1 - \frac{\mu_2}{\lambda_{1,s}}\Big)\|(u_n,v_n)\|^{2}_{Y}.
\end{align*}
Since $\mu_2 < \lambda_{1,s}$, the assertion follows.
\end{proof}

\noindent
The next result is useful to get nonnegative solutions as weak limits 
of Palais-Smale sequences. The same argument shows that a critical point of $J$
corresponds to a nonnegative solution to \eqref{eq:1.1}. 

\begin{lemma}\label{pos}
Assume that $b\geq 0$ and $\mu_2<\lambda_{1,s}$. Let $\{(u_n,v_n)\} \subset Y(\Omega)$ be a Palais-Smale sequence for
the functional $J$. Then
$$
\lim_n \|(u^-_n,v^-_n)\|_{Y}=0.
$$
In particular, the weak limit $(u,v)$ of the PS-sequence $\{(u_n,v_n)\}$ has nonnegative components.
\end{lemma}
\begin{proof}
By choosing $\varphi:=u^-\in X(\Omega)$ and 
$\psi:=v^-\in X(\Omega)$ as test functions in \eqref{deriv}
and using the elementary inequality
\begin{equation*}
(a-b)(a^- - b^-)\geq (a^- - b^-)^2,  \,\,\,\quad \text{for all $a,b\in\R$},
\end{equation*}
we obtain 
\begin{align*}
&\int_{\R^{2N}}\frac{(u(x)-u(y))(u^-(x)-u^-(y))+(v(x)-v(y))(v^-(x)-v^-(y))}{|x-y|^{N+2s}} dxdy \\
&\geq \int_{\R^{2N}}\frac{(u^-(x)-u^-(y))^2+(v^-(x)-v^-(y))^2}{|x-y|^{N+2s}} dxdy.
\end{align*}
Now, note that, since $b\geq 0$ and $w^-\leq 0$ and $w^+ \geq 0$, it holds
$$
\int_{\R^N}(A(u,v),(u^-,v^-))_{\R^2} \,dx \leq \int_{\R^N}(A(u^-,v^-),(u^-,v^-))_{\R^2} \,dx. 
$$
In fact, it follows
\begin{align*}
(A(u,v),(u^-,v^-))_{\R^2} &= (A(u^-,v^-),(u^-,v^-))_{\R^2} + b ((v^+)u^- + (u^+)v^-), \\
&\leq (A(u^-,v^-),(u^-,v^-))_{\R^2}.
\end{align*}
In turn, from the formula for $J'(u,v)(u^-,v^-)$, it follows that
\begin{align*}
J'(u,v)(u^-, v^-)& \geq \frac{C(N,s)}{2}\int_{\R^{2N}}\frac{(u^-(x)-u^-(y))^2+(v^-(x)-v^-(y))^2}{|x-y|^{N+2s}} dxdy\\
&- \int_{\Omega}(A(u^-,v^-),(u^-,v^-))_{\R^2} dx\geq I(u^-)+I(v^-),
\end{align*}
where we have set
$$
I(w):=\frac{C(N,s)}{2}\int_{\R^{2N}}\frac{(w(x)-w(y))^2}{|x-y|^{N+2s}} dxdy- \mu_2\int_{\Omega}|w|^2dx=[w]^2_s - \mu_2 \|w\|^2_{L^2(\Omega)}.
$$
On the other hand, by the definition of $\lambda_{1,s}$, we have
$$
I(w)\geq \Big(1-\frac{\mu_2}{\lambda_{1,s}}\Big)[w]^2_s,
$$
which finally  yields the inequality
$$
J'(u,v)(u^-, v^-)\geq \Big(1-\frac{\mu_2}{\lambda_{1,s}}\Big)\big([u^-]^2_s+[v^-]^2_s).
$$
Since $\{(u_n,v_n)\}\subset Y(\Omega)$ is a Palais-Smale sequence, we get
$J'(u_n,v_n)(u_n^-, v_n^-)=o_n(1)$, from which that assertion immediately follows.
\end{proof}

\noindent
From the boundedness of Palais-Smale sequences (see Lemma \ref{bounded}) and compact embedding theorems, passing to a subsequence if necessary,  there exists $(u_0,v_0)\in Y(\Omega)$
which, by Lemma \ref{pos}, satisfies  $u_0,v_0 \geq 0,$ such that
$(u_n,v_n)\rightharpoonup (u_0,v_0)$ weakly in $Y(\Omega)$ as $n\to\infty,$
$(u_n,v_n)\rightarrow (u_0,v_0)$ a.e.\ in $\Omega$ and strongly in $L^{r}(\Omega)$ for
 $1\leq r< 2^{*}_{s}$. Recalling that the sequences 
 $$
 w^{1}_n:=(u_n^+)^{p-1} (v_n^+)^q,\qquad 
 w^{2}_n:=(v_n^+)^{q-1} (u_n^+)^p, \quad p+q=2^*_s,
 $$ 
 are uniformly bounded
in $L^{(2^{*}_{s})'}(\Omega)$ and
converges pointwisely to  
$w^{1}_0=u_0^{p-1} v_0^q$ and
$w^{2}_0=v_0^{q-1} u_0^p$ 
respectively, we obtain 
$$
(w^{1}_n,w^{2}_n) \rightharpoonup (w^{1}_0,w^{2}_0),
\quad \text{weakly in $L^{(2^{*}_{s})'}(\Omega)$, as $n\to\infty$}.
$$
Hence, passing to the limit in
$$
J'(u_n,v_n)(\phi,\psi)=o_n(1), \quad \forall  (\phi,\psi)\in Y(\Omega),\quad 
\text{as $n\to \infty$},
$$
we infer that $(u_0,v_0)$ is a nonnegative weak solution.
Now, to conclude the proof, it is 
sufficient to prove that the solution is nontrivial.
\vskip4pt
\noindent
 {\sc Claim:} $(u_0,v_0)\neq (0,0).$
Notice that if $(u_0,v_0)$ is a solution of system with $u_0=0$, then $v_0=0.$ The same holds for the reversed situation.
\noindent
In fact, suppose $u_0=0$. Then, if $b>0$, it follows $v_0=0$. If,
instead, $b=0$, then 
$c \in \{\mu_1,\mu_2\}< \lambda_{1,s}.$ Since $v_0$ is a solution of equation
$$
(-\Delta)^s  v_0=c v_0 \quad \text{in $\Omega$}, \qquad  
v_0=0\ \text{on $\R^N\setminus\Omega$},
$$ we have that $v_0=0.$ Therefore, we may suppose that $(u_0,v_0)=(0,0).$
Define, as in \cite{bn},
$$
L:= \lim_{n\to \infty}
\int_{\R^{N}}\left(|(-\Delta)^{\frac{s}{2}} u_n|^2+|(-\Delta)^{\frac{s}{2}} v_n|^2\right)dx,
$$
from $J'(u_n,v_n)(u_n,v_n)=o_n(1),$
we get 
$$
\lim_{n\to \infty} \int_{\R^N} (u_n^+)^p(v_n^+)^q  dx= \frac{L}{2}.
$$
Recalling that $J(u_n,v_n)=c+ o_n(1),$ thus
\begin{equation}\label{cc}
c=\frac{sL}{N}.
\end{equation}
\noindent
From the definition of \eqref{quoSpqtil}, we have
\begin{equation}\nonumber
\displaystyle\int_{\R^{N}}\left(|(-\Delta)^{\frac{s}{2}} u_n|^2+|(-\Delta)^{\frac{s}{2}} v_n|^2\right)dx \geq \widetilde{\S}_s
\Big(\displaystyle\int_{\R^{N}}|u_n(x)|^{p}|v_n(x)|^{q}dx \Big)^{\frac{2}{2^*_s} }
\end{equation}
and passing to the limit the inequality above, we get
$$
L \geq \widetilde{\S}_s  \Big(\frac{L}{2}\Big)^{2/2^{*}_{s}}.
$$
Now, combining this estimate with (\ref{cc}), it follows that
\begin{equation}\label{cmais}
c\geq \frac{2s}{N} \Big(\frac{\widetilde{\S}_s}{2}\Big)^{N/2s}.
\end{equation}
Take $B,C>0$ with $B/C=\sqrt{p/q}$ and let $u_\epsilon\geq 0$ as in Proposition \ref{prop2}.
Fix $\eps>0$ sufficiently small that Proposition~\ref{scalar} holds and define
$v_{\epsilon}:=u_\epsilon/\|u_\epsilon\|_{L^{2_s^*}}$. Using the definition of $\S_{s,\lambda}(u)$, for every $t\geq 0$, we obtain
\begin{align*}
J&(tBv_\epsilon,tCv_\epsilon)\\
&\leq
\frac{t^2(B^2+C^2)}{2} \Big(\int_{\R^{N}}|(-\Delta)^{\frac{s}{2}} v_\epsilon|^2dx
- \mu_1\int_{\R^N}|v_\epsilon|^2dx \Big)-\frac{2t^{ 2^{*}_{s}}B^pC^q}{ 2^{*}_{s}}\\
&=\frac{t^2(B^2+C^2)}{2} \S_{s,\mu_1}(u_\epsilon)
-\frac{2t^{ 2^{*}_{s}}B^pC^q}{ 2^{*}_{s}}
:= \psi(t),\;  t\geq 0.
\end{align*}
Thus, an elementary calculation yields
$$
\psi_{\rm max} =\max_{\R^+}\psi =  \dfrac{2s}{N}\left\{\dfrac{(B^2+C^2) }{2(B^p C^q)^{2/ 2^{*}_{s}}}  \S_{s,\mu_1}(u_{\epsilon})\right\}^{N/2s}.
$$
By Lemma \ref{relation} and Proposition \ref{scalar}, we conclude that, for $\epsilon>0$ small,
\begin{align*}
\psi_{\rm max}&<\frac{2s}{N}\left\{\frac{(B^2+C^2) }{2(B^p C^q)^{2/ 2^{*}_{s}}}\S_s\right\}^{N/2s}\\
&=\frac{2s}{N}\left\{\frac{1}{2}\Big[\Big(\frac{p}{q}\Big)^{q/p+q}+\Big(\frac{p}{q}\Big)^{-p/p+q}\Big]\S_s\right\}^{N/2s} \\
&=\frac{2s}{N}\Big(\frac{\widetilde{\S}_s}{2}\Big)^{N/2s}.
\end{align*}
Let now $\gamma\in C([0,1],Y(\Omega))$ be defined by
$$
\gamma(t):=(\tau tBv_\epsilon,\tau tCv_\epsilon),\quad t\in [0,1],
$$
where $\tau>0$ is sufficiently large that $J(\tau Bv_\epsilon,\tau Cv_\epsilon)\leq 0$. Hence, $\gamma\in\Gamma$ and we conclude that
\begin{align*}
c\leq\sup_{t\in [0,1]}J(\gamma(t))\leq\sup_{t\geq 0}J(tBv_\epsilon,tCv_\epsilon)\leq \psi_{\rm max} < \frac{2s}{N}\Big(\frac{\widetilde{\S}_s}{2}\Big)^{N/2s},
\end{align*}
which contradicts \eqref{cmais}.
Hence $(u_0,v_0) \neq (0,0)$ and the proof is complete. 
Finally, that $u_0>0$ and $v_0>0$ follows as in the sub-critical case. \qed

\bigskip

\end{document}